\documentclass{gtmon_a}
\pdfoutput=1

\usepackage{lamsarrow}
\usepackage{pb-diagram}
\usepackage{pb-lams}

\proceedingstitle{Proceedings of the Nishida Fest (Kinosaki 2003)}
\conferencestart{28 July 2003}
\conferenceend{8 August 2003}
\conferencename{International Conference in Homotopy Theory}
\conferencelocation{Kinosaki, Japan}

\editor{Matthew Ando}
\givenname{Matthew}
\surname{Ando}

\editor{Norihiko Minami}
\givenname{Norihiko}
\surname{Minami}

\editor{Jack Morava}
\givenname{Jack}
\surname{Morava}

\editor{W Stephen Wilson}
\givenname{W Stephen}
\surname{Wilson}

\title[Spectra with unusual invariant ideals]{On the homotopy groups of $E(n)$--local spectra\\ with unusual invariant ideals}

\author{Hirofumi Nakai}
\givenname{Hirofumi}
\surname{Nakai}
\address{Department of Mathematics\\
Faculty of Technology\\
Musashi Institute of Technology\\\newline 
Tokyo 158-8557\\Japan
}
\email{nakai@ma.ns.musashi-tech.ac.jp}
\urladdr{}

\author{Katsumi Shimomura}
\givenname{Katsumi}
\surname{Shimomura}
\address{Department of Mathematics\\
Faculty of Science\\
Kochi University\\\newline
Kochi 780-8520\\Japan}
\email{katsumi@math.kochi-u.ac.jp}
\urladdr{}

\volumenumber{10}
\issuenumber{}
\publicationyear{2007}
\papernumber{18}
\startpage{319}
\endpage{332}

\doi{}
\MR{}
\Zbl{}

\arxivreference{}


\received{31 August 2004}
\revised{16 September 2005}
\accepted{}
\published{18 April 2007}
\publishedonline{18 April 2007}
\proposed{}
\seconded{}
\corresponding{}
\version{}

\makeatletter
\def\cnewtheorem#1[#2]#3{\newtheorem{#1}{#3}[section]
\expandafter\let\csname c@#1\endcsname\c@thm}

\numberwithin{equation}{section}

\newtheorem{thm}{Theorem}[section]
\cnewtheorem{lem}[thm]{Lemma}
\cnewtheorem{prop}[thm]{Proposition}
\cnewtheorem{cor}[thm]{Corollary}
\cnewtheorem{conj}[thm]{Conjecture}

\theoremstyle{definition}
\cnewtheorem{defin}[thm]{Definition}
\cnewtheorem{exm}[thm]{Example}
\cnewtheorem{rem}[thm]{Remark}
\cnewtheorem{exer}[thm]{Exercise}
\cnewtheorem{ques}[thm]{Question}
\cnewtheorem{nota}[thm]{Notation}
\cnewtheorem{form}[thm]{Formula}
\cnewtheorem{cond}[thm]{Condition}
\cnewtheorem{meth}[thm]{Method}
\let\c@equation\c@thm

\makeatother  

\makeop{Hom}
\makeop{Ext}
\makeop{Tor}
\makeop{Cotor}
\makeop{Ker}
\makeop{Coker}
\makeop{Image}

\makeop{mog}

\newcommand{\ox}{\otimes}
\newcommand{\cotensor}{\Box}

\newcommand{\Frac}[2]{\displaystyle\frac{#1}{#2}}

\newcommand{\ints}{\mathbb{Z}}

\makeop{Spin}

\begin{document}

\begin{htmlabstract}
Let E(n) and T(m) for nonnegative integers n and m
denote the Johnson--Wilson and the Ravenel spectra, respectively.  
Given a spectrum whose E(n)<sub>*</sub>&ndash;homology is
E(n)<sub>*</sub>(T(m))/(v<sub>1</sub>,&hellip;,v<sub>n-1</sub>),
then each homotopy group of it estimates the order of each homotopy
group of L<sub>n</sub>T(m).  We here study the E(n)&ndash;based Adams
E<sub>2</sub>&ndash;term of it and present that the determination of the
E<sub>2</sub>&ndash;term is unexpectedly complex for odd prime case. At
the prime two, we determine the E<sub>&infin;</sub>&ndash;term for
&pi;<sub>*</sub>(L<sub>2</sub>T(1)/(v<sub>1</sub>)), whose computation
is easier  than that of &pi;<sub>*</sub>(L<sub>2</sub>T(1)) as we expect.
\end{htmlabstract}

\begin{abstract}
Let $E(n)$ and $T(m)$ for nonnegative integers $n$ and $m$ denote
the Johnson--Wilson and the Ravenel spectra, respectively.  Given a spectrum
whose $E(n)_{*}$--homology is $E(n)_*(T(m))/(v_1,\ldots,v_{n-1})$, then
each homotopy group of it estimates the order of each homotopy group
of $L_nT(m)$.  We here study the $E(n)$--based Adams $E_2$--term of it and
present that the determination of the $E_2$--term is unexpectedly complex
for odd prime case. At the prime two, we determine the $E_{\infty}$--term
for $\pi_*(L_2T(1)/(v_1))$, whose computation is easier  than that of
$\pi_*(L_2T(1))$ as we expect.
\end{abstract}

\maketitle

\section{Introduction}

In \cite{r:book}, Ravenel has constructed
the homotopy associative commutative ring spectrum $T(m)$ 
as a summand of $p$--component of the Thom spectrum associated with the map
$\Omega SU(p^m) \to BU$.
It is extensively used in \cite[Section 7]{r:book}
to compute the homotopy groups of spheres in terms of
``the method of infinite descent''.
The Adams--Novikov $E_{2}$--term converging to the stable homotopy groups
$\pi_*(T(m))$
is described  by use of the Hopf algebroid
$(BP_{*},\Gamma(m+1))$ (cf~\cite[Definition 7.1.1]{r:book}).
In particular, the $0$--th line is
\[
\Ext_{\Gamma(m+1)}^{0}(BP_{*},BP_{*})=\ints_{(p)}[v_{1}, \ldots ,
v_{m}]\subset BP_*=\ints _{(p)}[v_{1}, \ldots ],
\]
and the more the value of $m$, the more primitives we obtain.
Since  $v_{k}$ for $1 \le k \le m$ 
is a permanent cycle of the spectral sequence, we obtain spectra $T(m)/(v_k)$ and $T(m)/(v_k,v_l)$ for $1 \le k,l \le m$ (see \fullref{selfmap-T}.)
Here $T(m)/J$ for an ideal $J$ of $BP_*$ denotes a spectrum such that $BP_*(T(m)/J)=BP_*/J$.

Let $BP\langle n\rangle$ denote the Johnson--Wilson ring spectrum with
$BP\langle n\rangle_*=\ints_{(p)}[v_{1}, \ldots ,$ $v_{n}]$ 
and put $E(n)=v_n^{-1}BP\langle n\rangle$ as usual. 
Then we have the $E(n)$--based Adams spectral sequence $E_r^{s,t}(X)\Rightarrow \pi_*(L_nX)$ for a spectrum $X$, 
whose $E_2$--term is 
$E_2^*(X)=\Ext_{E(n)_*(E(n))}^*$ $(E(n)_*,E(n)_*(X))$.  
Here $L_n$ denotes the Bousfield localization functor with respect to $E(n)$.
Note that $BP_*(T(m))=BP_*[t_1,\ldots,t_m]\subset BP_*[t_1,\ldots]=BP_*(BP)$.
In order to study the $E_2$--term for a spectrum $X$ with
$E(n)_*(X)=E(n)_*/J[t_1,\ldots,t_m]$ for an ideal $J$ of $E(n)_*$, we introduce the generalized Johnson--Wilson spectrum $E_m(n)=v_n^{-1}BP\langle n+m\rangle$. 
Then $$\Sigma(n,m)=E_m(n)_*\ox_{BP_*}BP_*[t_{m+1}, t_{m+2},\ldots]\ox_{BP_*}E_m(n)_*$$ is a Hopf algebroid over $E_m(n)_*$, and
 the $E(n)$--based Adams $E_2$--term $E_2^*(X)$ is isomorphic to $\Ext_{\Sigma(n,m+1)}^*(E_m(n)_*,E_m(n)_*/J)$, which we denote $\Ext^*(E_m(n)_*/J)$, by a similar change-of-rings theorem of Hovey and Sadofsky \cite{hs}.

Consider $J_n$ be the sequence $v_1,v_2,\ldots, v_{n-1}$. 
Then $T(m)/(J_n)$ exists if $n\le 2$ as commented above.
Besides, if $L_nT(m)/J$ exists, then the $E(n)$--based Adams $E_2$--term for $\pi_*(L_nT(m)/J)$ is isomorphic to an Ext group $\Ext^*(E(n)_*/J)$.
Consider the long exact sequence of Ext groups associated to the short exact sequence 
\[
0
\arrow{e}
E_{m}(n)_{*}/(J_{n})
\arrow{e}
p^{-1}E_{m}(n)_{*}/(J_{n})
\arrow{e}
E_{m}(n)_{*}/(p^{\infty},J_{n})
\arrow{e}
0.
\]
Since $\Ext^*(p^{-1}E_{m}(n)_{*}/(J_{n}))=\Q$, \fullref{Ext0_n<=m} implies our first theorem:

\begin{thm}\label{int Ext0_n<=m}
The $\Ext$ group
$\Ext^0(E_m(n)_*/(J_{n}))$ is isomorphic to
$\ints_{(p)}$, and the group $E_2^1(E_m(n)_*/(J_{n}))$ is
isomorphic to the direct sum of
the cyclic module over the ring $\ints_{(p)}[v_{n}^{\pm 1},v_{n+1}, \ldots , v_{m}]$
generated by 
$$\frac{v_{m+1}^{e_{1}} \ldots v_{m+n}^{e_{n}}}{p^{1+\nu(e_{k})}}$$
of order $\smash{p^{1+\nu(e_{k})}}$
with 
$\nu(e_{k})=\min\{ \nu(e_{1}), \ldots , \nu(e_{n}) \}$,
where the integer $\nu(\ell)$ denotes the maximal power of $p$ 
that divides $\ell$.  
\end{thm}
For the case where $n>m$, we have an example which has a similar result
to the above theorem
(cf~\fullref{example_0dim_n>m}):

\begin{prop}\label{int 2}
The $E(2)$--based Adams $E_2$--term
$\smash{E_2^0(T(1)/(v_1))}$
is isomorphic to $\ints_{(p)}$
and 
$E_2^1(T(1)/(v_1))$ is the direct sum of
the cyclic module over $\ints_{(p)}$ generated by
$\smash{v_{2}^{\smash[b]{sp^i}}v_{3}^{\smash[b]{tp^j}}/p^{\smash[b]{1+ \min(i,j)}}}$
of order $\smash{p^{\smash[b]{1+ \min(i,j)}}}$.
\end{prop}

In these cases, we did not determine $E_2^s$ for $s>1$ since there is an
obstruction, which comes from the generators known as $b_{i,j}$ (see
\eqref{bij}).
This is what we did not expect.
For $p=2$, we have the relation $b_{i,j}=h_{i,j}^2$, which makes possible
to compute for $s>1$.
Since the $E(2)$--based Adams differentials are read off from Mahowald and Shimomura \cite{ms}, we obtain
the $E_\infty$--term.
\eject
\begin{thm}\label{intro3}
Let $p=2$.
The $E(2)$--based Adams $E_{\infty}$--term 
for $\pi_{*}(L_{2}T(1)/(v_{1}))$
is isomorphic to $\ints_{(2)}$ if $s=0$ and is isomorphic to the tensor
product of 
$\Lambda(\rho_{2})$
and the direct sum of
\begin{enumerate}
\item
$\widetilde{v_{2}A[h_{20}]}$,
$v_{3}B[h_{30}]/(h_{30}^{3})$
and 
$v_{3}Bh_{30}h_{31}$
whose elements are of order two,
\item
$M^{0}$
and
$M^{1}$. 
\end{enumerate}
Here the modules are given in \fullref{Shimomura_p=2}.
\end{thm}

In \fullref{generalized-JW}, we consider  the Hopf algebroid
$(E_{m}(n)_{*},\Sigma(n,m+1))$ and show a variation of the change-of-rings theorem given in Hovey and Sadofsky \cite{hs}.
In \fullref{MapsSpectra}, we exhibit the formulas for
the structure maps 
(the right unit $\eta_{R}$ and the diagonal maps $\Delta$).
We then observe the existence of spectra of the form $T(m)/J$.
\fullref{Ext_n<=mcase} is devoted to prove \fullref{int Ext0_n<=m} and \fullref{int 2}.
In \fullref{Shimomura_p=2}, we determine
the $E_{\infty}$--term for $\pi_{*}(L_{2}T(1)/(2^{\infty},v_{1}))$.
The homotopy groups $\pi_*(L_2T(1))$ is determined easily if $p$ is odd, and
stays undetermined if $p=2$. The result of this section is the first step to
understand $\pi_*(L_2T(1))$ at the prime two.

{\bf Acknowledgements}\qua
We wish to thank to the organizers of Nishida Conference held on August 2003 
for making arrangement of the publication for the Proceedings. 
We are also grateful to Ippei Ichigi for reading the draft paper carefully 
and for pointing out some misprints. 

\section{A generalized Johnson--Wilson theory}\label{generalized-JW}

Let $BP$ and $BP\langle n\rangle$ denote the Brown--Peterson and the Johnson--Wilson spectra characterized by
$\pi_{*}(BP)
     = BP_* = \ints_{(p)}[v_{1}, \ldots , v_{n} , \ldots]$ and
$\pi_{*}(BP\langle n\rangle)
     = BP\langle n\rangle_*= \ints_{(p)}[v_{1}, \ldots , v_{n}]
           \subset 
           BP_{*}
$
with 
$|v_{n}|=|t_{n}|=2(p^{n}-1)$.
Then the $BP_*$--homology of $BP$ is
$BP_{*}(BP)= BP_{*}[t_{1}, \ldots , t_{n} , \ldots]$,
We put 
$$
E_m(n)=v_n^{-1}BP\langle n+m \rangle
$$
for nonnegative integers $n$ and $m$.
Then 
$$
E_{m}(n)_{*} = E(n)_{*}[v_{n+1}, \ldots ,v_{n+m}] \subset v_{n}^{-1}BP_{*}
.$$
We notice that $E_0(n)$ is the localized Johnson--Wilson spectrum $E(n)$.

Let $\Gamma(m+1)$ (cf Ravenel~\cite[7.1.1]{r:book})
be the $BP_{*}(BP)$--comodule defined by
\[
\Gamma(m+1)
= BP_{*}(BP)/(t_{1}, \ldots , t_{m})
= BP_{*}[t_{m+1},t_{m+2}, \ldots] .
\]
Then the pair $(BP_{*},\Gamma(m+1))$ has the structure of the Hopf
algebroid
inherited from $(BP_{*},BP_{*}(BP))$.
Put 
\[
\Sigma_{m}(n,i)
=
E_{m}(n)_{*} \ox_{BP_{*}} \Gamma(i) \ox_{BP_{*}} E_{m}(n)_{*}.
\]
In particular, 
we write
$$
    \Sigma(n,m+1)=\Sigma_{m}(n,m+1)=E_{m}(n)_{*} \ox_{BP_{*}} \Gamma(m+1) \ox_{BP_{*}} E_{m}(n)_{*}.
$$
The pair 
$(E_{m}(n)_{*},\Sigma_{m}(n,i))$
is a Hopf algebroid
with the structure maps inherited from those of the Hopf algebroid
$(BP_{*}, \Gamma(i))$ for all $i>0$.
Consider the map between Hopf algebroids
$
(E_{m}(n)_{*},\Sigma_{m}(n,1))
\arrow{e}
(E_{m}(n)_{*},\Sigma(n,m+1))
$
induced from the projection
from $BP_{*}(BP)$ to $\Gamma(m+1)$.
The map is normal and that
\begin{equation}
E_{m}(n)_{*}(T(m)) = E_{m}(n)_{*} \cotensor_{\Sigma(n,m+1)}\Sigma_{m}(n,1)
\label{key-equation}
\end{equation}
if $m>0$. 
Here, $T(m)$ denotes the
Ravenel spectrum \cite[6.5.1]{r:book},
which is an associative commutative ring spectrum characterized by
$BP_{*}(T(m))= BP_{*}[t_{1}, \ldots , t_{m}]$.
Since $\Sigma_{m}(n,1)$ is $E_{m}(n)_{*}(E_{m}(n))$,
the change-of-rings theorem \cite[A1.3.12]{r:book} shows the following:

\begin{lem}\label{change-of-rings-1}
There is an isomorphism
$$\Ext_{E_{m}(n)_{*}(E_{m}(n))}(E_{m}(n)_{*},E_{m}(n)_{*}(T(m)))
= \Ext_{\Sigma(n,m+1)}(E_{m}(n)_{*},E_{m}(n)_{*}).
$$
\end{lem}

\begin{rem}
In general, equation \eqref{key-equation} does not hold
if we work on $E(n)_*E(n)$--comodules.
For example, if we set $(n,i)=(2,3)$, then
$$
\Sigma_{0}(2,3)
= E(2)_{*} [t_{3},t_{4}, \ldots]/(\eta_{R}(v_{k})\, \colon\, k>2)   
.$$
In the right hand side we have the relation
$v_{2}t_{1}^{p^2} \equiv v_{2}^{p}t_{1}$ mod $(p)$
since $\eta_{R}(v_{3})=0$.
On the other hand, we do not have any relation on $t_{1}$ in
$
E(2)_{*}T(2)= E(2)_{*}[t_{1},t_{2}] .
$
\end{rem}

Since $E_{m}(n)_{*}$ is a free $E(n)_{*}$--module over the bases
$\smash{v^{E}= v_{n+1}^{e_{1}} \ldots v_{n+m}^{e_{m}}}$
for
$E=(e_{1}, \ldots , e_{m})$ with $e_{k} \ge 0$,
there is a homotopy equivalence
$E_{m}(n)= \bigvee_{E} \Sigma^{|E|}E(n)$.
This shows that the $E(n)$--based and
the $E_{m}(n)$--based Adams spectral sequences
agrees from the $E_{2}$--term (cf Hovey and Sadofsky~\cite{hs}).

\section{Existence of some spectra}\label{MapsSpectra}

An ideal $I=(a_0,a_1,\ldots, a_{n-1})$ of $BP_*$ is called \textit{invariant} 
if $\eta_R(a_i)\equiv a_i$ mod $(a_0,a_1,\ldots, a_{i-1})$ 
for each $0\le i< n$ as a $BP_*BP$--comodule.
It is well known that
if there is a spectrum $X$ such that $BP_*(X)=BP_*/I$, then $I$ is invariant.
Consider now the Ravenel spectrum $T(m)$. 
Then the $E_2$--term of the Adams--Novikov spectral sequence for $\pi_*(W\wedge T(m))$ for a spectrum $W$ is isomorphic to an Ext group over the Hopf algebroid $(BP_*, \Gamma(m+1))$.
We call an ideal $J=(w_0,w_1,\ldots, w_{n-1})$ of $BP_*$ \textit{unusual} 
if it is not invariant and $\eta_R(w_i)\equiv w_i$ mod $(w_0,w_1,\ldots, w_{i-1})$ 
for each $0\le i< n$ as a $\Gamma(m+1)$--comodule.
In the same manner as above, if there is a spectrum $X$ such that
$BP_*(X)=BP_*/J[t_1,\ldots,t_m]$ for $m>0$, then $J$ is invariant or unusual.
In this section, we study the existence of a spectrum $X$ with $BP_*$--homology (resp.\ $E(n)_*$--homology)
\[
BP_{*}(X)=
BP_{*}/J[t_{1}, \ldots , t_{m}]\quad ({\rm resp.}\ E(n)_{*}(X)=
E(n)_{*}/J[t_{1}, \ldots , t_{m}])
\]
for an unusual ideal $J$.
We write $T(m)/J$ (resp.\ $L_nT(m)/J$) for such $X$.

The next lemma is verified
by Hazewinkel's and Quillen's formulas (see Miller, Ravenel and Wilson
\cite[(1.1)--(1.3)]{mrw}):

\begin{lem}\label{structuremaps_n<=m}
Assume that $n \le m$.
Let $J_{n}$ denote the ideal $(v_{1}, \ldots , v_{n-1})$ of $BP_{*}$.
Then the structure maps in $(BP_{*},\Gamma(m+1))$
act as
\[\begin{array}{rlll}
\eta_{R}(v_{k})
& \equiv &v_{k}
         &  
           \mbox{for $n \le k \le m$,} \\
\eta_{R}(v_{m+k})
& \equiv  &v_{m+k}+pt_{m+k}
         &  
           \mbox{for $0 < k \le n$,} \\
\Delta(t_{m+k})
& \equiv  &t_{m+k} \ox 1
         + 1 \ox t_{m+k}
         &
         \mbox{for $0 \le k \le n$,} \\
\Delta(t_{m+n+1})
& \equiv & t_{m+n+1} \ox 1
         + 1 \ox t_{m+n+1}
         + v_{n}b_{m+1,n-1}&
\end{array}\]
{\rm mod} $J_{n}$, 
where
\begin{equation}
b_{i,j}= 
\big(
  t_{i}^{p^{j+1}} \ox 1
+ 1 \ox t_{i}^{p^{j+1}}
- ( t_{i} \ox 1 
  + 1 \ox t_{i} )^{p^{j+1}}
\big)/p.
\label{bij}
\end{equation}
\end{lem}

By this lemma, we read off the behavior
of the structure maps $\eta_{R}$ and $\Delta$ mod $J_{n}$
of the Hopf algebroid
$(E_{m}(n)_{*},\Sigma(n,m+1))$.
For $n>m$, we only consider the case where $n=2$ and $m=1$.

\begin{lem}\label{structuremaps_n2m1}
The structure maps in $(BP_{*},\Gamma(2))$ acts as
\begin{eqnarray*}
\eta_{R}(v_{i})
& \equiv & v_{i}
         + pt_{i} 
           \quad
           \mbox{for $i=2$ and $3$}    , \\
\eta_{R}(v_{4})
& \equiv & v_{4}
         + v_{2}t_{2}^{p^2}
         + pt_{4}
         + v_{2}c_{21}
         - \eta_{R}(v_{2})^{p^2}t_{2}  , \\
\eta_{R}(v_{5})
& \equiv & v_{5}
         + v_{3}t_{2}^{p^3}
         + v_{2}t_{3}^{p^2}
         + pt_{5}
         + v_{2}c_{31}
         + v_{3}c_{22}
         - \eta_{R}(v_{3})^{p^2}t_{2}
         - \eta_{R}(v_{2})^{p^3}t_{3}  , \\
\Delta(t_{i})
& \equiv & t_{i} \ox 1
         + 1 \ox t_{i}
           \quad
           \mbox{for $i=2$ and $3$}    , \\
\Delta(t_{4})
& \equiv & t_{4} \ox 1
         + 1 \ox t_{4}
         + t_{2} \ox t_{2}^{p^2}
         + v_{2}b_{21}                 , \\
\Delta(t_{5})
& \equiv & t_{5} \ox 1
         + 1 \ox t_{5}
         + t_{3} \ox t_{2}^{p^3}
         + t_{2} \ox t_{3}^{p^2}
         + v_{2}b_{31}
         + v_{3}b_{22}
\end{eqnarray*}
{\rm mod} $(v_{1})$,
where 
$c_{i,j}= p^{-1}(v_{i}^{p^{j+1}}-\eta_{R}(v_{i}^{p^{j+1}}))$.
In particular, 
$b_{i,j} \equiv t_{i}^{2^j} \ox t_{i}^{2^j}$
{\rm mod} $(2)$
for $p=2$.
\end{lem}

We consider the Adams--Novikov spectral sequence
\begin{equation}
E_2^{*,*}(X)=\Ext_{BP_*(BP)}^{*,*}(BP_*,BP_*(X))
\quad \Longrightarrow \quad
\pi_*(X).
\label{ANSS_T(m)}
\end{equation}
By the change-of-rings theorem \cite[A1.3.12]{r:book},
we have an isomorphism
\begin{equation}
E_2^{*}(T(m)/I_{n})= \Ext_{\Gamma(m+1)}^{0}(BP_{*}/I_{n}).
\label{e2Tm/In}
\end{equation}
Hereafter we use the abbreviation:
$$
\Ext_{\Gamma}(A,-)=\Ext_{\Gamma}(-)\quad\mbox{for a Hopf algebroid $(A,\Gamma)$.}
$$
\fullref{structuremaps_n<=m} implies the following:

\begin{lem}\label{elements_E2}
For $0 \le k \le m$,
$$
v_{n+k} \in E_2^{0}(T(m)/I_{n})=\Ext_{\Gamma(m+1)}^{0}(BP_{*}/I_{n}),$$
where
$I_{n}=(p)+J_n$.
\end{lem}

\begin{lem}\label{selfmap-T}
Let $M$ be a $T(m)$--module spectrum.
If 
$\alpha$ and $\beta \in E_{2}(T(m))$ are permanent cycles in the spectral sequence \eqref{ANSS_T(m)}, 
then
there exist spectra of the form
$M/(\alpha^{a})$
and 
$M/(\alpha^{a},\beta^{b})$
for positive integers $a$ and $b$.  
In particular, we have
$T(m)/(v_{k}^{a})$
and 
$\smash{T(m)/(v_{i}^{a}, v_{j}^{b})}$ for $i,j,k< m+2$.  
\end{lem}

\begin{proof}
Since $M$ is a $T(m)$--module spectrum, the elements $\alpha$ and $\beta$ 
yield the self maps on $M$, which we also denote by $\alpha$ and $\beta$.
Now $M/(\alpha^{a})$ is a cofiber of the self map $\alpha^a$, 
and the $M/(\alpha^{a},\beta^{b})$ is obtained 
by use of Verdier's axiom on the equation 
$\alpha^{a}\beta^{b}=\beta^{b}\alpha^{a}$ 
in $[M,M]_*$.

Since the reduced comodule $\overline{\Gamma(m+1)}$
is $(2p^{m+1}-3)$--connected, we have the vanishing line $E_{2}^{s,t}(T(m))=0$ for $t<2s(p^{m+1}-1)$ by  \eqref{e2Tm/In}.
It follows that 
$v_{k}\in E_2^*(T(m))$ in \fullref{elements_E2} is permanent if $k<m+2$.
\end{proof}

The existence of a spectrum with $BP_{*}$--homology
$BP_{*}/I_{n}$ is problematic
and we still have little information
for such a spectrum,
which we usually call the ({$(n{-}1)$st}) Smith--Toda spectrum
and is denoted by $V(n-1)$
(eg Smith \cite{smith}, Toda \cite{toda} and Ravenel~\cite{r:book}).
For $n \le 3$,
it is shown that 
$V(n)$ exists if and only if $p>2n$.
On the other hand, $L_nV(n-1)$ exists if $n^2+n<2p$ \cite{s-y}.
The smash products $T(m)$ and these Smith--Toda spectra show the following:

\begin{prop}
If $p>2n$, $T(m)/I_n$ exists, and if  $n^2+n<2p$,
$L_nT(m)/I_{n}$ exists.  
\end{prop}

\section[Computing the connecting homomorphisms]{$\Ext_{\Sigma(n,m+1)}^s(E_{m}(n)_{*}/J_{n})$ for small $s$}
\label{Ext_n<=mcase}

In this section, let $J_{n}$ denote the sequence $v_1,\ldots,v_{n-1}$ of elements of $E_m(n)_*$.
Applying Ext to the short exact sequence
\[
0
\arrow{e}
E_{m}(n)_{*}/(p,J_{n})
\arrow{e,t}{1/p}
E_{m}(n)_{*}/(p^{\infty},J_{n})
\arrow{e,t}{p}
E_{m}(n)_{*}/(p^{\infty},J_{n})
\arrow{e}
0,
\]
we have the long exact sequence of Ext groups 
with connecting homomorphism 
\begin{equation}
\delta \co \Ext^{r}(E_{m}(n)_{*}/(p^{\infty},J_{n})) \ 
\arrow{e}\ 
\Ext^{r+1}(E_{m}(n)_{*}/(p,J_{n})).
\label{connectinghomo}
\end{equation}
By \cite[Theorem 6.5.6]{r:book}, we know the structure of
$\Ext(E_{m}(n)_{*}/(p,J_{n}))$, which means that 
$\Ext(E_{m}(n)_{*}/(J_{n}))$
is a computable object.

To compute 
$\Ext(E_{m}(n)_{*}/(p^{\infty},J_{n}))$,
we redefine the class $h_{m+k,0}$ $(0<k \le n)$ by
\begin{equation}
h_{i,0}
= 
\biggl[
\Frac{\log (1+pv_{i}^{-1}t_{i})}{p}
\biggr]
= 
\biggl[\,
\sum_{n>0\vphantom{y}}(-1)^{n-1}
\Frac{(pv_{i}^{-1}t_{i})^{n}}{p n}
\biggr]                 .
\label{redef_hi0}
\end{equation}

\begin{lem}\label{cocycle_h0}
For $0<k \le n$, the connecting homomorphism $\delta$ in
\eqref{connectinghomo} acts for all $\ell$ as
$\delta(h_{m+k,0}/p^{\ell})=0$.
\end{lem}

\begin{proof}
It suffices to show that
$ph_{m+k,0}=d(\log(v_{m+k}))$.
By \fullref{structuremaps_n<=m}, we have
$
\eta_{R}(v_{m+k})=v_{m+k}+pt_{m+k}
$
for $0<k \le n$, so the equation
$$\log(1+pv_{m+k}^{-1}t_{m+k})
= \log(\eta_{R}(v_{m+k}))-\log(v_{m+k})
= d(\log(v_{m+k}))$$
holds.
\end{proof}

The element $\smash{v_{m+k}^{k+1}x}$ is well-defined in $\Sigma(n,m+1)/(p^k)$,
although the representative $\smash{x=\log(1+pv_{m+k}^{-1}t_{m+k})/p}$
of $h_{m+k,0}$ has negative exponents of $v_{m+k}$ in the coefficient.

An easy computation with \fullref{structuremaps_n<=m} shows the following:

\begin{lem}\label{dif-0dim}
Put 
$\nu(e_{k})=\min\{ \nu(e_{1}), \ldots , \nu(e_{n}) \}$.
Then we have 
\[
\delta
\biggl(
\Frac{v_{m+1}^{e_{1}} \ldots v_{m+n}^{e_{n}}}{p^{1+\nu(e_{k})}}
\biggl)
= 
v_{m+1}^{e_{1}} \ldots v_{m+n}^{e_{n}}h_{m+k,0}
+ \cdots
\]
in $\Ext^{1}(E_{m}(n)_{*}/(p,J_{n}))$
up to unit.
For $\nu$, see \fullref{int Ext0_n<=m}.
\end{lem}
\eject
\begin{cor}\label{Ext0_n<=m}
$\Ext^{0}(E_{m}(n)_{*}/(p^{\infty},J_{n}))$ is the direct sum of
\begin{enumerate}
\item
the cyclic $\ints_{(p)}[v_{n}^{\pm 1},v_{n+1}, \ldots , v_{m}]$--module
generated by
$$\Frac{v_{m+1}^{e_{1}} \ldots v_{m+n}^{e_{n}}}{p^{1+\nu(e_{k})}}$$
of order $p^{1+\nu(e_{k})}$
with 
$\nu(e_{k})=\min\{ \nu(e_{1}), \ldots , \nu(e_{n}) \}$
and 
\item
$\Q/\ints_{(p)}[v_{n}^{\pm 1},v_{n+1}, \ldots , v_{m}]$.
\end{enumerate}
\end{cor}

\begin{exm}\label{example_0dim}
For $m=n=2$, we have
\[
\delta\biggl(\Frac{v_{3}^{sp^i}v_{4}^{tp^j}}{p^{1+\min(i,j)}}\biggr)=
\begin{cases}
v_{3}^{sp^i}v_{4}^{tp^j} h_{40}     & \mbox{for $i>j$}  \\
v_{3}^{sp^i}v_{4}^{tp^j} h_{30}     & \mbox{for $i<j$}  \\
v_{3}^{sp^i}v_{4}^{tp^j}
( h_{30} + a h_{40} )               & \mbox{for $i=j$}
\end{cases}
\]
in $\Ext^{1}(E_{2}(2)_{*}/(p,v_{1}))$
up to unit (where $a \in (\ints/(p))^{\times}$),
and 
$\Ext^{0}(E_{2}(2)_{*}/(p^{\infty},v_{1}))$
is the direct sum of
\begin{enumerate}
\item
the cyclic module over $\ints_{(p)}[v_{2}^{\pm}]$ generated by
$\unfrac{v_{3}^{sp^i}v_{4}^{tp^j}}{p^{1+\min(i,j)}}$
of order $\smash{p^{1+\min(i,j)}}$
and 
\item
$\Q/\ints_{(p)}[v_{2}^{\pm 1}]$.
\end{enumerate}
In the computations for $\delta(h_{31})$ and $\delta(h_{41})$,
the elements $b_{i,j}$ (cf~\fullref{structuremaps_n<=m})
occur, which are hard to express
in terms of generators appearing in \cite[Theorem 6.5.6]{r:book}.
We observe that the specific property $b_{i,j}=h_{i,j}^{2}$ at $p=2$ 
makes the computations easy.
\end{exm}

We consider the spectrum $L_{n}T(m)/(J_{n})$ for
$(n,m)=(2,1)$, 
which is the simplest case satisfying $n>m$,
and compute $\Ext_{\Sigma(2,2)}^s(E_{1}(2)_{*}/(v_{1}))$ for $s<2$ for an
odd prime.
We consider the case for $p=2$ in the next \fullref{Shimomura_p=2}.
Since  $p$ is odd, the condition of \cite[Theorem 6.5.6]{r:book}
is always satisfied and
$\Ext_{\Sigma(2,2)}(E_{1}(2)_{*}/(p,v_{1}))$
is obtained as
\[
K(2)_{*}[v_{3}]
\ox 
\Lambda(h_{i,j}\, \colon\, 2 \le i \le 3, j \in \ints/2).
\]
Starting from this,
$\Ext_{\Sigma(2,2)}^{0}(E_{1}(2)_{*}/(p^{\infty},v_{1}))$
is determined by computing the connecting homomorphism \eqref{connectinghomo}
for $(m,n)=(1,2)$ as follows \fullref{Ext0_n<=m}:

\begin{prop}\label{example_0dim_n>m}
For 
$sp^i \in \ints$
and 
$tp^j \ge 0$, 
we have
\[
\delta(v_{2}^{sp^i}v_{3}^{tp^j}/p^{1+ \min(i,j)})
= 
\left\{
\begin{array}{ll}
s v_{2}^{sp^i-1}v_{3}^{tp^j}h_{20}    & \mbox{if $i<j$} , \\
t v_{2}^{sp^i}v_{3}^{tp^j-1}h_{30}    & \mbox{if $i>j$} , \\
v_{2}^{sp^i-1}v_{3}^{tp^j-1}
(s v_3h_{20} + t v_2h_{30})                 & \mbox{if $i=j$,}   \\
\end{array}
\right.
\]
and 
$\Ext_{\Sigma(2,2)}^{0}(E_{1}(2)_{*}/(p^{\infty},v_{1}))$
is the direct sum of
\begin{enumerate}
\item
the cyclic $\ints_{(p)}$--module generated by
$v_{2}^{sp^i}v_{3}^{tp^j}/p^{1+ \min(i,j)}$
of order $p^{1+ \min(i,j)}$ and
\item
$\Q/\ints_{(p)}$.
\end{enumerate}
\end{prop}

\section[Computation of an example at the prime two]{The homotopy groups $\pi_*(L_{2}T(1)/(v_{1}))$ at the prime two}
\label{Shimomura_p=2}

We begin with recalling the result of Mahowald and Shimomura
 \cite{ms}:
\begin{equation}
\Ext(E_{1}(2)_{*}/(2,v_{1}))
=
K(2)_{*}[v_{3},h_{20}]
\ox 
\Lambda(h_{21},h_{30},h_{31},\rho_{2})
\label{result_MahowaldShimomura}
\end{equation}
where 
$\rho_{2}$
is the generator of degree $0$ represented by the cocycle
$v_{2}^{-5}t_{4} + v_{2}^{-10}t_{4}^{2}$.
We see that \eqref{redef_hi0} for $p=2$
is also a cocycle with leading term $v_{i}^{-2}t_{i}^{2}$,
and replace the representative cocycles by
\[ h_{i,0} = [ t_{i} ] \quad \mbox{and} \quad h_{i,1} =
\biggl[ \sum_{n>0\vphantom{y}}(-1)^{n-1} \Frac{(2v_{i}^{-1}t_{i})^{n}}{2 n}
\biggr].  \]
Setting $B=\ints/2[v_{2}^{\pm 2},v_{3}^{2}]$,
we rewrite the right hand side of \eqref{result_MahowaldShimomura} as
\[
B
\ox 
\Lambda(v_{3})
\ox 
\Lambda(h_{21},h_{30},h_{31})
\ox 
\Lambda(v_{2})
\ox 
\ints/2[h_{20}]
\ox 
\Lambda(\rho_{2})  .
\]
\[
h_{21}h_{31}
      = v_{2}^{-1}v_{3}^{2}h_{20}h_{21}
         + v_{2}^{2}h_{30}^{2}
         + v_{2}h_{20}h_{31} \leqno{\hbox{Since}}
\]
by \cite[p\ 243 (1)]{ms},
we replace 
$h_{21}h_{31}$ (resp.\ $h_{21}h_{30}h_{31}$)
with 
$h_{30}^{2}$ (resp.\ $h_{30}^{3}$).

\begin{lem}
As the $\ints/2$--module,
$\Ext(E_{1}(2)_{*}/(2,v_{1}))$
is isomorphic to 
\begin{align*}
&A \ox \Lambda(v_{2}) \ox \ints/2[h_{20}] \ox \Lambda(\rho_{2}) \\
\text{where}\qquad &A     =  B 
\ox \Lambda(v_{3}) \ox \bigl( \ints/2[h_{30}]/(h_{30}^{4}) \oplus
  \ints/2\{ h_{21} , h_{31} \} \ox \Lambda(h_{30}) \bigr) \\
\text{and}\qquad &B     =  \ints/2[a_{2}^{\pm 1},a_{3}]
  \qquad \text{with } a_{i}=v_{i}^{2}.
\end{align*}
\end{lem}

\begin{lem}
The connecting homomorphism 
\eqref{connectinghomo} for $(m,n)=(1,2)$
acts as 
$$\delta(v_{i}^{s}/2)=v_{i}^{s-1}h_{i,0}
\quad\text{and}\quad
\delta(a_{i}^{2^n s}/2^{n+2})=a_{i}^{2^n s}h_{i,1}\qquad(i=2,3)$$
for odd $s$ and $n \ge 0$.
\end{lem}

\proof
It follows from 
\[
\begin{array}{rcll}
\hspace{60pt}(v_{i}^{s}) 
& \equiv & 2 v_{i}^{s-1}t_{i}
           & \mod{(4)}         , \\
d(a_{i}^{2^{n} s}) 
& \equiv & 2^{n+2} v_{i}^{2^{n+1} s}
           ( v_{i}^{-1}t_{i}
           + v_{i}^{-2}t_{i}^{2} )
           & \mod{(2^{n+3})}   .\hspace{58pt}\qed
\end{array}
\]
\medskip

$\Ext(E_{1}(2)_{*}/(2,v_{1}))$ is decomposed
into the following four summands tensoring with $\Lambda (\rho_2)$:
\begin{align*}
&v_{2}A \oplus \Lambda(v_{2}) \ox A \ox \ints/p(h_{20})h_{20} \\
&v_{3}B \oplus \Lambda(v_{3}) \ox B \{ h_{30}, h_{30}^{2}, h_{30}^{3} \}
  \oplus v_{3}h_{30}h_{31} B \\
&B \oplus B \left\{ h_{21} , h_{31} \right\} \oplus v_{3}h_{21}h_{30} B \\
&B h_{30} \left\{ h_{21} , h_{31} \right\} \oplus v_{3}B 
  \left\{ h_{21} , h_{31} \right\} 
\end{align*}
With respect to each summand, we construct a long exact sequence
in \fullref{summand_1}, \fullref{summand_2} and \fullref{summand_3}.
We often use the replacement
\[ h_{31} = [v_{3}^{-1}t_{3} + v_{3}^{-2}t_{3}^{2}]
= v_{3}^{-1}h_{30} + \cdots.  \]
If we define $P_{i}$ $(i \ge 0)$ and $Q_{j}$ $(j>0)$ by
\begin{eqnarray*}
P_{i} & = & \ints_{(2)} \bigl\{ a_{2}^{2^i s}a_{3}^{2^j t} \,\colon\,
           0 \le j \le i, 0 \neq s \in \ints, t \ge 0 \bigr\}    , \\
Q_{j} & = & \ints_{(2)} \bigl\{ a_{2}^{2^i s}a_{3}^{2^j t} \,\colon\,
           0 \le i < j, s \in \ints, t>0 \bigr\}    ,
\end{eqnarray*}
then we decompose $B$ into 
\[ B= \bigg( \bigoplus_{i \ge 0} P_{i} \bigg)
  \oplus \bigg( \bigoplus_{j>0} Q_{j} \bigg) .  \]
Define $M^{0}$ and $M^{1}$ by 
\begin{eqnarray*}
M^{0} & = & \biggl( \bigoplus_{i \ge 0} P_{i}
           \Bigl\{ \Frac{1}{2^{i+2}} \Bigr\} \biggr)
           \oplus \biggl( \bigoplus_{j>0} Q_{j}
           \Bigl\{ \Frac{1}{2^{j+2}} \Bigr\} \biggr)
           \oplus \Q/\ints_{(2)} , \\
M^{1}
     & = & \biggl( \bigoplus_{i \ge 0} P_{i}
           \Bigl\{ \Frac{h_{21}}{2^{i+2}} \Bigr\} \biggl)
           \oplus \biggl( \bigoplus_{j>0} Q_{j}
           \Bigl\{ \Frac{h_{31}}{2^{j+2}} \Bigr\} \biggr)             .
\end{eqnarray*}
Then we have the following results: 

\begin{lem}\label{summand_1}
We have two long exact sequences
$$\begin{diagram}
  \node{B}
  \arrow{e,V}
    \node{M^{0}}
    \arrow{e,t}{2}
      \node{M^{0}}
      \arrow{wsw,b}{\delta}     \\
  \node{B \{ h_{21}, h_{31} \}}
  \arrow{e}
    \node{M^{1}}
    \arrow{e,t}{2}
      \node{M^{1}}
      \arrow{wsw,b,A}{\delta}   \\
  \node{v_{3}h_{21}h_{30} B}
\end{diagram}$$
and
$$\begin{diagram}
  \node{v_{3}B \{ h_{21}, h_{31} \}}
  \arrow{e,V}
    \node{(v_{3}/2)B \{ h_{21}, h_{31} \}}
    \arrow{e,t}{2}
      \node{(v_{3}/2)B \{ h_{21}, h_{31} \}}
      \arrow{wsw,b,A}{\delta}          \\
  \node{B h_{30} \{ h_{21}, h_{31} \} .}
\end{diagram}$$
\end{lem}

\begin{proof}
In the first sequence the connecting homomorphism acts as: 
\begin{eqnarray*}
\delta( a_{2}^{2^{i}s} a_{3}^{2^{j}t} /2^{2+\min{(i,j)}})
     & = & \left\{
           \begin{array}{ll}
           a_{2}^{2^{i}s}a_{3}^{2^{j}t}h_{21}  & (i<j)  \\
           a_{2}^{2^{i}s}a_{3}^{2^{j}t}h_{31}  & (i>j)  \\
           a_{2}^{2^{i}s}a_{3}^{2^{j}t}
           (h_{21} + h_{31})                   & (i=j)  \\
           \end{array}
           \right.     \\
\end{eqnarray*}
We also see that 
$\smash{\delta(a_{2}^{2^{i}s}a_{3}^{2^{j}t}h_{31}/2^{i+2})} \text{ for } i<j,$ 
$\smash{\delta(a_{2}^{2^{i}s}a_{3}^{2^{j}t}h_{21}/2^{j+2})} \text{ for } i>j,$
and $\smash{\delta(a_{2}^{2^{i}s}a_{3}^{2^{i}t}h_{21}/2^{i+2})}$
are equal to $a_{2}^{2^{i}s}a_{3}^{2^{j}t}h_{21}h_{31}$. 
Replacing $h_{31}$ with $v_{3}^{-1}h_{30}+ \cdots$, 
we have the first sequence. 
The second sequence is obvious. 
\end{proof}

\begin{lem}\label{summand_2}
We have a long exact sequence 
\[
\begin{diagram}\dgARROWLENGTH= .5em
  \node{v_{3}B}
  \arrow{e,V}
    \node{(v_{3}/2)B}
    \arrow{e,t}{2}
      \node{(v_{3}/2)B}
      \arrow{wsw,t}{\delta}     \\
  \node{h_{30}B \ox \Lambda(v_{3})}
  \arrow{e}
    \node{(v_{3}h_{30}/2)B}
    \arrow{e,t}{2}
      \node{(v_{3}h_{30}/2)B}
      \arrow{wsw,t}{\delta}     \\
  \node{\begin{array}{l}h_{30}^{2}B \ox \Lambda(v_{3}) \\ \oplus 
(v_{3}h_{30}h_{31}) B\end{array}}
  \arrow{e}
    \node{\begin{array}{l}(v_{3}h_{30}^{2}/2)B \\ \oplus 
(v_{3}h_{30}h_{31}/2)B\end{array}}
    \arrow{e,t}{2}
      \node{\begin{array}{l}(v_{3}h_{30}^{2}/2)B \\ \oplus 
(v_{3}h_{30}h_{31}/2)B\end{array}}
      \arrow{wsw,b,A}{\delta}   \\
  \node{h_{30}^{3}B \ox \Lambda(v_{3}).}
\end{diagram}
\]
\end{lem}

\begin{proof}
It follows from 
$$\eqalignbot{
\delta(a_{2}^{2^{i}s}a_{3}^{2^{j}t}v_3h_{30}^{k}/2)
     & =  a_{2}^{2^{i}s}a_{3}^{2^{j}t}h_{30}^{k+1} 
           \qquad
           \text{for } 0 \le k \le 2,                    \cr
\delta(a_{2}^{2^{i}s}a_{3}^{2^{j}t}v_3h_{30}h_{31}/2)
     & =  a_{2}^{2^{i}s}a_{3}^{2^{j}t}h_{30}^{2}h_{31}
      =  a_{2}^{2^{i}s}a_{3}^{2^{j}t-1}v_3h_{30}^{3}
         + \cdots                                        
}\proved$$
\end{proof}

\begin{lem}\label{summand_3}
We have a long exact sequence
\[
\begin{diagram}
  \node{v_{2}A}
  \arrow{e,V}
    \node{(v_{2}/2)A}
    \arrow{e,t}{2}
      \node{(v_{2}/2)A}
      \arrow{wsw,t}{\delta}     \\
  \node{h_{20}A \ox \Lambda(v_{2})}
  \arrow{e}
    \node{(v_{2}h_{20}/2)A}
    \arrow{e,t}{2}
      \node{(v_{2}h_{20}/2)A}
      \arrow{wsw,t}{\delta}     \\
  \node{h_{20}^{2}A \ox \Lambda(v_{2})}
  \arrow{e}
    \node{(v_{2}h_{20}^{2}/2)A}
    \arrow{e,t}{2}
      \node{\cdots.}   
\end{diagram}
\]
\end{lem}

\begin{proof}
Notice that each exponent of $v_{2}$
in $(v_{2}h_{20}^{k}/2)A$ is odd. 
Since we have $d(x)=0$ for $x \in A$ in the cobar complex, 
we have 
\begin{eqnarray*}
d(v_{2}^{2s+1}v_{3}^{t}x)
     & = & d(v_{2}^{2s+1}v_{3}^{t}) \ox x  . 
\end{eqnarray*}
We see that 
\[
d(v_{2}^{2s+1}v_{3}^{t})= 
\left\{
\begin{array}{ll}
2v_{2}^{2s}v_{3}^{2n}t_{2}
+ \cdots   & \mbox{for $t=2n$}  ,  \\
d(v_{2}^{2s+1}v_{3}^{t})= 2v_{2}^{2s}v_{3}^{2n}(v_{3}t_{2}+v_{2}t_{3})
+ \cdots   & \mbox{for $t=2n+1$}.  
\end{array}
\right.
\] 
In both cases we obtain 
\begin{eqnarray*}
\delta 
\left(
\Frac{v_{2}^{2s+1}v_{3}^{t} x}{2}
\right)
     & = & v_{2}^{2s}v_{3}^{t}h_{20} x  
\end{eqnarray*}
replacing $v_{3}h_{20}$ by 
$v_{3}h_{20}=[v_{3}t_{2}+v_{2}t_{3}]$
only for the case $t=2n+1$. 
\end{proof}

By the above three lemmas, we obtain the chart of differentials 

$
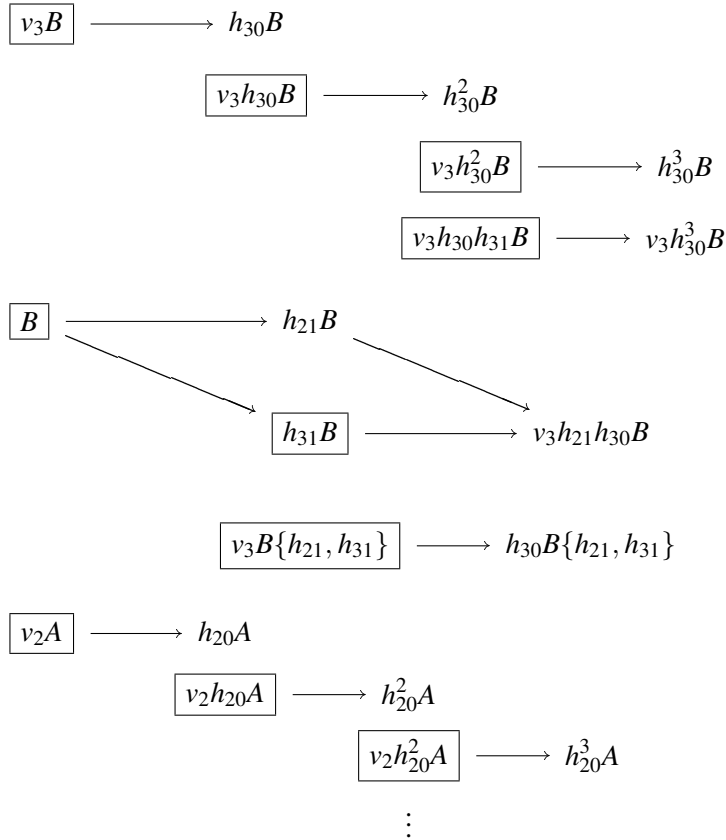
\begin{diagram}
\node{\fbox{$v_{3}B$}}
\arrow{e}
  \node{h_{30}B}               \\
  \node[2]{\fbox{$v_{3}h_{30}B$}}
  \arrow{e}
    \node{h_{30}^{2}B}         \\
    \node[3]{\fbox{$v_{3}h_{30}^{2}B$}}
    \arrow{e}
      \node{h_{30}^{3}B}       \\
    \node[3]{\fbox{$v_{3}h_{30}h_{31}B$}}
    \arrow{e}
      \node{v_{3}h_{30}^{3}B}  
\end{diagram}$

$\begin{diagram}
\node{\fbox{$B$}}
\arrow{e}
\arrow{se}
  \node{h_{21}B}
  \arrow{se}                   \\
  \node[2]{\fbox{$h_{31}B$}}
  \arrow{e}
    \node{v_{3}h_{21}h_{30}B}  \\
  \node[2]{\fbox{$v_{3}B \{ h_{21}, h_{31} \}$}}
  \arrow{e}
    \node{h_{30}B \{ h_{21}, h_{31} \}}  
\end{diagram}$

$\begin{diagram}
\node{\fbox{$v_{2}A$}}
\arrow{e}
  \node{h_{20}A}               \\
  \node[2]{\fbox{$v_{2}h_{20}A$}}
  \arrow{e}
    \node{h_{20}^{2}A}         \\
    \node[3]{\fbox{$v_{2}h_{20}^{2}A$}}
    \arrow{e}
      \node{h_{20}^{3}A}       \\
    \node[3]{\vdots}           
\end{diagram}$

Thus we conclude the following:

\begin{lem}\label{Ext_p=2}
$\Ext_{\Sigma(2,2)}(E_{1}(2)_{*},E_{1}(2)_{*}/(2^{\infty},v_{1}))$
is the tensor product of 
$\Lambda(\rho_{2})$
and the direct sum of 

\begin{enumerate}
\item
$v_{2}A[h_{20}]$, 
$v_{3}B[h_{30}]/(h_{30}^{3})$
and 
$v_{3}Bh_{30}h_{31}$
whose elements are of order two, \item
$M^{0}$
and
$M^{1}$.
\end{enumerate}
\end{lem}

Let $E_{\infty}^{*}(X)$ for a spectrum $X$ denote the $E_{\infty}$--term of 
the $E(2)$--based
Adams spectral sequence converging to the homotopy groups 
$\pi_{*}(L_2X)$.

\begin{thm}\label{homotopygroup_L2T1/2inftyv_1}
The $E_{\infty}$--term $E_{\infty}^{*}(L_2T(1)/(2^{\infty}, v_1))$ is
the tensor product of 
$\Lambda(\rho_{2})$
and the direct sum of 

\begin{enumerate}
\item
$\widetilde{v_{2}A[h_{20}]}$, 
$v_{3}B[h_{30}]/(h_{30}^{3})$
and 
$v_{3}Bh_{30}h_{31}$
whose elements are of order two,
\item
$M^{0}$
and
$M^{1}$, 
\end{enumerate}
where $\widetilde{v_{2}A[h_{20}]}$ denotes the module 
\[
\bigl(
\ints/2[v_{2}^{\pm 2}, v_{3}^{4}] 
\ox \Lambda(v_{3})
\ox
\bigl(
\ints/2[h_{30}]/(h_{30}^{4}) \oplus \ints/2\{ h_{21},h_{31} \} \ox 
\Lambda(h_{30})
\bigr)
\bigr)
[h_{20}]/(h_{20}^{3}) .
\]
\end{thm}

\begin{proof}
In \cite{ms}, the 
differentials of $E(2)$--based Adams  spectral sequence for 
$L_{2}T(1)/I_{2}$ (written as $D$ in \cite{ms}) are determined as 
\[
d_{3}(v_{3}) = 0
\quad
\mbox{and}
\quad
d_{3}(v_{3}^{k}) = v_{2}^{2}v_{3}^{k-2}h_{20}^{3}
\quad
\mbox{for $2 \le k \le 3$,}
\]
and 
$d_{3}(v_{3}^{k}x)=d_{3}(v_{3}^{k})x$
for $x= h_{20}$, $h_{21}$, $h_{30}$ and $h_{31}$.
Note that for each element $w a_{3}^{2t+1} \in v_{2}A[h_{20}]$, 
we see that 
\[
d_{3}(w a_{3}^{2t+1}/2)
= w a_{3}^{2t}h_{20}^{3}/2 \in v_{2}A[h_{20}].  
\]
This shows the structure of $\pi_{*}(L_{2}T(1)/(2^{\infty},v_{1}))$, 
since it has a horizontal vanishing line. 
\end{proof}

\begin{proof}[Proof of \fullref{intro3}]
Consider the cofiber sequence 
\[
\begin{diagram}
\node{T(1)/(v_1)}
\arrow{e}
  \node{T(1)/(v_1) \wedge S\Q}
  \arrow{e}
    \node{T(1)/(2^{\infty},v_1) .}
\end{diagram}
\]
Then the homotopy groups of 
$T(m)/(v_1) \wedge S\Q$ 
and 
$T(1)/(2^{\infty},v_1)$ are determined in 
\cite[Corollary 6.5.6]{r:book} and \fullref{homotopygroup_L2T1/2inftyv_1}, 
respectively.
\end{proof}

\bibliographystyle{gtart}
\bibliography{link}

\begin{thebibliography}{}
\providecommand\bibmarginpar{\leavevmode\marginpar}
\def\urlstyle#1{{\tt #1}}

\bibitem{hs}
\textbf{M Hovey}, \textbf{H Sadofsky},
  \href{http://dx.doi.org/10.1112/S0024610799007784} {\emph{Invertible spectra
  in the $E(n)$--local stable homotopy category}}, J. London Math. Soc. $(2)$
  60 (1999) 284--302 \xox{MR}{1722151}

\bibitem{ms}
\textbf{M Mahowald}, \textbf{K Shimomura}, \emph{The Adams--Novikov spectral
  sequence for the $L_2$ localization of a $v_2$ spectrum}, from: ``Algebraic
  topology (Oaxtepec, 1991)'', Contemp. Math. 146, Amer. Math. Soc.,
  Providence, RI (1993)  237--250 \xox{MR}{1224918}

\bibitem{mrw}
\textbf{H\,R Miller}, \textbf{D\,C Ravenel}, \textbf{W\,S Wilson},
  \emph{Periodic phenomena in the Adams--Novikov spectral sequence}, Ann. Math.
  $(2)$ 106 (1977) 469--516 \xox{MR}{0458423}

\bibitem{r:book}
\textbf{D\,C Ravenel}, \emph{Complex cobordism and stable homotopy groups of
  spheres}, Pure and Applied Mathematics 121, Academic Press, Orlando, FL
  (1986) \xox{MR}{860042}

\bibitem{s-y}
\textbf{K Shimomura}, \textbf{Z-I Yosimura}, \emph{$\mathrm{BP}$--{H}opf module
  spectrum and $\mathrm{BP}_{*}$--{A}dams spectral sequence}, Publ. Res. Inst.
  Math. Sci. 22 (1986) 925--947 \xox{MR}{866663}

\bibitem{smith}
\textbf{L Smith},
  \href{http://links.jstor.org/sici?sici=0002-9327(197101)93:1%3C226:ORCBMI%3E%
2.0.CO%3B2-1} {\emph{On realizing complex bordism modules. {A}pplications to
  the stable homotopy of spheres}}, Amer. J. Math. 92 (1970) 793--856
  \xox{MR}{0275429}

\bibitem{toda}
\textbf{H Toda}, \emph{On spectra realizing exterior parts of the {S}teenrod
  algebra}, Topology 10 (1971) 53--65 \xox{MR}{0271933}

\end{thebibliography}

\end{document}